\documentclass[12pt,reqno]{article}

\usepackage[T1]{fontenc}
\usepackage[utf8]{inputenc}
\usepackage{lmodern}
\usepackage{authblk}
\usepackage{amsmath,amssymb,amsthm,mathtools}
\usepackage{enumitem}
\usepackage{booktabs}
\usepackage{array}
\usepackage{hyperref}
\usepackage{microtype}
\hypersetup{hidelinks}

\newtheorem{theorem}{Theorem}[section]
\newtheorem{lemma}[theorem]{Lemma}

\newtheorem{maintheorem}{Theorem}

\DeclareMathOperator{\Aut}{Aut}

\DeclareMathOperator{\GL}{GL}

\DeclareMathOperator{\PSL}{PSL}

\title{A rigidity theorem for skew braces with multiplicative group \(S\times T\)}
\author{
Marco Damele\thanks{The author was supported by INdAM and GNSAGA -- Gruppo Nazionale per le Strutture Algebriche, Geometriche e le loro Applicazioni, and by ProBiki, funded by Fondazione di Sardegna.}
}

\affil{
Dipartimento di Matematica, Università degli Studi di Cagliari,
Via Ospedale 72, 09124 Cagliari, Italy \\
\texttt{marco.damele@unica.it}
}
\date{}
\begin{document}
\maketitle
\vspace{-2.0 em}

\begin{abstract}
We prove that if \(B\) is a finite skew brace with
\((B,\cdot)\cong S\times T\), where \(S\) and \(T\) are non-abelian simple
groups, then \((B,+)\) is not supersolvable.
\end{abstract}


\medskip


\medskip

\section{Introduction}

Skew braces were introduced by Guarnieri and Vendramin in \cite{17} as a
group-theoretic framework for the study of set-theoretic solutions of the
Yang--Baxter equation. A \emph{skew brace} is a triple $(B,+,\cdot)$
such that \((B,+)\) and \((B,\cdot)\) are groups and
\[
a\cdot (b+c)=a\cdot b-a+a\cdot c
\qquad\text{for all } a,b,c\in B.
\]
The groups \((B,+)\) and \((B,\cdot)\) are called, respectively, the
\emph{additive group} and the \emph{multiplicative group} of the skew
brace.

A central problem in the theory is to understand the relation between
these two groups. More precisely, one would like to determine which pairs
of groups can occur as the additive and multiplicative groups of the same
skew brace, and how structural properties of one group impose restrictions
on the other. In this direction, Smoktunowicz and Vendramin proved in
\cite{27} that if \(B\) is finite and \((B,+)\) is nilpotent, then
\((B,\cdot)\) is solvable. This result is related to a question of Byott,
originally arising in the context of Hopf--Galois theory
\cite{byott2014solubilitycriteriahopfgaloisstructures}, and later
formulated by Vendramin in terms of skew braces \cite{33}: if \(B\) is a
finite skew brace and \((B,+)\) is solvable, must \((B,\cdot)\) also be
solvable? This is often referred to as the Byott--Vendramin conjecture. Several partial results are known. For instance, the conjecture has been
verified for various classes of finite skew braces; see, among others,
\cite{16,24,9,21, 11, DameleZGroup, 30}. Byott also studied the case in which the
multiplicative group is a non-abelian simple group \cite{8} and later Tsang in \cite{31} studied the case where $(B,\cdot)$ is quasisimple. In this paper
we consider the next natural situation, namely skew braces whose
multiplicative group is a direct product of two non-abelian simple groups.

Our main result is the following rigidity theorem.

\begin{maintheorem}\label{ThA}
Let \(B\) be a finite skew brace such that
$(B,\cdot)\cong S\times T,$ where \(S\) and \(T\) are finite non-abelian simple groups. Then
$(B,+)$ is not supersolvable.
\end{maintheorem}

The proof combines several ingredients. First, assuming that \((B,+)\) is
supersolvable, one obtains a characteristic Hall subgroup of odd order in
the additive group, which becomes a \(2\)-complement in the multiplicative
group. This forces each simple factor \(S\) and \(T\) to have a subgroup of
index a power of \(2\). By the classification result used in \cite{16},
the factors are then reduced to groups of projective linear type. A further
result of Byott implies that the multiplicative group has a section
isomorphic to \(\operatorname{PSL}_2(7)\). Using a simple lemma on sections
of direct products, this section must come from one of the two simple
factors. This reduces the possible multiplicative group to
\[
\operatorname{PSL}_2(7)\times \operatorname{PSL}_2(p),
\]
with \(p+1\) a power of \(2\). The remaining cases are excluded by
analyzing suitable Sylow subgroups of the additive group and the induced
lambda action.

The paper is organized as follows. In Section~\ref{Preliminaries} we
collect the preliminary results needed in the proof. In
Section~\ref{sec:thmA} we prove the main theorem.

\section{Preliminaries}\label{Preliminaries}

In this section we collect the few facts that will be used in the proof of
the main theorem. 
Let \(B=(B,+,\cdot)\) be a skew brace. Recall that, for every \(b\in B\),
the map
\[
\lambda_b:(B,+)\longrightarrow (B,+),
\qquad
a\mapsto -b+b\cdot a,
\]
is an automorphism of the group \((B,+)\). Moreover, the assignment
\[
\lambda:(B,\cdot)\longrightarrow \Aut(B,+),
\qquad
b\mapsto \lambda_b,
\]
is a group homomorphism. A subgroup \(I\) of \((B,+)\) is called a \emph{left ideal} of \(B\) if it
is normal in \((B,+)\) and $\lambda_b(I)\leq I$ for every $b \in B.$ In particular, every characteristic subgroup of \((B,+)\) is a left ideal.
If \(I\) is a left ideal, then \((I,\cdot)\) is a subgroup of
\((B,\cdot)\). Given any skew brace $B$, one can replace the additive group $(B,+)$ by its \emph{opposite group} to obtain a new skew brace 
\[
B^{\mathrm{opp}} = (B,+^{\mathrm{opp}},\cdot),
\]
called the \emph{opposite skew brace} of $B$ (see \cite[Proposition 3.1]{koch2019oppositeskewleftbraces}).

We shall also use the following standard terminology. A \emph{section} of
a group \(G\) is a quotient \(M/N\), where \(N\trianglelefteq M\leq G\).

\begin{lemma}\label{subquotient}
Let \(S,T\) be finite groups and suppose that \(S\times T\) has a section
isomorphic to a non-abelian simple group \(L\). Then \(L\) is isomorphic
to a section of \(S\) or to a section of \(T\).
\end{lemma}

\begin{proof}
Let \(M\leq S\times T\) and \(N\trianglelefteq M\) be such that $M/N\cong L.$ Set
\[
K_S=M\cap (S\times 1),
\qquad
K_T=M\cap (1\times T).
\]
Since \(K_S\trianglelefteq M\) and \(K_T\trianglelefteq M\) we have that \(K_SN/N\) and \(K_TN/N\) are normal subgroups of \(M/N\). Since
\(M/N\cong L\) and \(L\) is simple, each of these two normal subgroups is either
trivial or equal to \(M/N\). Moreover, the subgroups \(K_S\) and \(K_T\) commute elementwise, because
\(K_S\leq S\times 1\) and \(K_T\leq 1\times T\). Hence their images
\(K_SN/N\) and \(K_TN/N\) commute elementwise in \(M/N\). Since \(L\) is
non-abelian, these two images cannot both be equal to \(M/N\). Assume first that $K_SN/N=M/N.$ Then \(M=K_SN\). By the second isomorphism theorem, we obtain
\[
M/N=K_SN/N\cong K_S/(K_S\cap N).
\]
Since \(K_S\leq S\times 1\cong S\), the quotient \(K_S/(K_S\cap N)\) is a
section of \(S\). Thus \(L\) is isomorphic to a section of \(S\). Similarly, if $K_TN/N=M/N,$
then \(M=K_TN\), and again by the second isomorphism theorem,
\[
M/N=K_TN/N\cong K_T/(K_T\cap N).
\]
Since \(K_T\leq 1\times T\cong T\), this shows that \(L\) is isomorphic to a
section of \(T\). It remains to consider the case $K_SN/N=1$ and $K_TN/N=1.$
Equivalently, $K_S\leq N$ and $K_T\leq N.$
Let $\pi_S:S\times T\to S$ and $\pi_T:S\times T\to T$
be the two canonical projections. We restrict them to \(M\). Then
\[
\ker(\pi_S|_M)=M\cap (1\times T)=K_T\leq N,
\]
and
\[
\ker(\pi_T|_M)=M\cap (S\times 1)=K_S\leq N.
\]

We claim that \(\pi_S|_M\) induces an isomorphism
\[
M/N\cong \pi_S(M)/\pi_S(N).
\]
Indeed, since \(\pi_S(N)\trianglelefteq \pi_S(M)\), the map
\[
\overline{\pi}_S:M/N\longrightarrow \pi_S(M)/\pi_S(N),
\qquad
mN\longmapsto \pi_S(m)\pi_S(N),
\]
is well defined and surjective. Its kernel is
\[
\ker \overline{\pi}_S
=
\{\,mN\in M/N : \pi_S(m)\in \pi_S(N)\,\}.
\]
If \(mN\in \ker\overline{\pi}_S\), then there exists \(n\in N\) such that $\pi_S(m)=\pi_S(n).$
Hence $mn^{-1}\in \ker(\pi_S|_M)=K_T\leq N.$
Since \(n\in N\), it follows that \(m\in N\). Therefore $\ker\overline{\pi}_S=1.$
Thus
\[
M/N\cong \pi_S(M)/\pi_S(N),
\]
which is a section of \(S\). Analogously, using the projection \(\pi_T\), one obtains
\[
M/N\cong \pi_T(M)/\pi_T(N),
\]
which is a section of \(T\). Therefore, also in this last case, \(L\) is isomorphic to a section of \(S\)
and to a section of \(T\). Hence, in all cases, \(L\) is isomorphic to a
section of \(S\) or to a section of \(T\), as required.
\end{proof}
The following result will be used to force the occurrence of
\(\operatorname{PSL}_{2}(7)\).

\begin{lemma}\label{L:Byott-section}
Let \(B\) be a finite skew brace such that \((B,+)\) is solvable and
\((B,\cdot)\) is not solvable. Then \((B,\cdot)\) has a section isomorphic
to \(\operatorname{PSL}_{2}(7)\).
\end{lemma}

\begin{proof}
This is the consequence of Byott's result used in
\cite[Corollary~3.3]{9}.
\end{proof}

We shall need the following elementary observation.

\begin{lemma} \label{index}
Let $n\in \mathbb N_{>0}$. Then $G=\operatorname{PSL}_{2}(7)^n$ has no subgroup of index $3^m$ with $1\leq m\leq n$.
\end{lemma}

\begin{proof}
Write $G=G_1\times \cdots \times G_n$ where $G_i\cong \operatorname{PSL}_{2}(7)$ 
for each $i$. Let $H\leq G$ be a subgroup such that $[G:H]=3^m$. If $G_i\leq H$ for every $i$, then
\[
G=\langle G_1,\dots,G_n\rangle\leq H,
\]
and hence $H=G$, a contradiction. Therefore $G_j\nleq H$ for some $j$. It follows that
$G_j\cap H < G_j,$
so $[G_j:G_j\cap H]>1$. On the other hand, $[G_j:G_j\cap H]=[G_jH:H]$
so this index divides $[G:H]=3^m$. Hence $[G_j:G_j\cap H]$ is a non-trivial power of $3$. This is impossible, since $\operatorname{PSL}_{2}(7)$ has no subgroups of index equal to a non-trivial power of $3$. Therefore no such subgroup $H$ exists.
\end{proof}

Finally, we record the comparison result between the derived series of the
additive and multiplicative groups.

\begin{lemma} \label{Key}
Let $(B,+,\cdot)$ be a finite skew brace such that $(B,+)$ is solvable, and let $H$ be a characteristic subgroup of $(B,+)$. If $\Aut((B,+)/H)$ is solvable, then there exists $k\in \mathbb N_{>0}$ such that
\[
(B,\cdot)^{(k)}\leq H.
\]
\end{lemma}

\begin{proof}
Since $(B,+)$ is solvable, there is $n\in \mathbb N_{>0}$ such that $(B,+)^{(n)}\leq H$. By \cite[Corollary~3.2]{21} there is $m\in \mathbb N_{>0}$ such that $(B,\cdot)^{(n+m)}\leq H$.
\end{proof}

\section{Proof of Theorem \ref{ThA}} \label{sec:thmA}

\begin{proof}[Proof of Theorem \ref{ThA}]
Suppose, for contradiction, that the statement is false, and let $B$ be a counterexample. Write $|B|=2^\alpha m,$
with $\alpha\in \mathbb N_{>0}$ and $m\in \mathbb N_{>0}$ odd. Since $(B,+)$ is supersolvable, \cite[Theorem~9.1]{18} ensures the existence of a characteristic subgroup $H\leq (B,+)$ of order $m$. Hence $(H,\cdot)$ is a $2$-complement of $(B,\cdot)$. Write $(B,\cdot)=G_1\times G_2,$
where $G_{1} \simeq S$ and $G_{2} \simeq T$. Let $i \in \{1,2\}$. Since each $G_i$ is normal in $(B,\cdot)$, we have $G_iH\leq (B,\cdot)$, and
$[G_i:G_i\cap H]=[G_iH:H]$ is a power of $2$. If \(G_i\leq H\) for some \(i\), then \(G_i\) has odd order. By the
Feit--Thompson theorem (see the main Theorem of \cite{FeitThompson1963}), \(G_i\) is solvable, a contradiction. Therefore $G_j\nleq H$ for every $j$, and so $G_j$ has a subgroup of index a power of $2$. Since $G_j\cap H\leq H$, it follows that $G_j\cap H$ is a $2$-complement of $G_j$. Thus \(S\) and \(T\) are non-abelian simple groups admitting a
\(2\)-complement. By \cite[Lemma~2.3]{16}, this implies that
\[
S\cong \PSL_n(q),
\qquad\text{where } \frac{q^n-1}{q-1} = 2^{t} 
\]
for some $t \in \mathbb{N}$ and 
\[
T\cong \PSL_m(q'),
\qquad\text{where } \frac{(q')^m-1}{q'-1} = 2^{t'}
\]
for some \(t'\in \mathbb N\). By \cite[Corollary~3.3]{9}, $(B,\cdot)$ has a section isomorphic to $\operatorname{PSL}_{2}(7)$. By Lemma \ref{subquotient} we can assume  without loss of generality that $S$ also has a section isomorphic to $\operatorname{PSL}_{2}(7)$. Write $q=p^f$ for a prime $p$ and a positive integer $f$. Reducing modulo $2$ we see that $p$ has to be odd. Dividing both sides by $p^f-1$, we obtain
\[
\frac{p^{fn}-1}{p^f-1}=2^t.
\]
The left-hand side is a geometric series of $n$ terms with ratio $p^f$:
\[
(p^f)^{n-1}+(p^f)^{n-2}+\cdots+p^f+1=2^t.
\]
Since $p^f$ is odd, each term of the sum is odd. Hence, in order for the sum to be a power of $2$, the number of terms $n$ must be even. Let $n=2k$. Then
\[
p^{2fk}-1=(p^{fk}-1)(p^{fk}+1),
\]
so that the equation becomes
\[
\frac{(p^{fk}-1)(p^{fk}+1)}{p^f-1}=2^t
\qquad\Longrightarrow\qquad
\left(\frac{p^{fk}-1}{p^f-1}\right)(p^{fk}+1)=2^t.
\]
For the product of two integers to be a power of $2$, both factors must themselves be powers of $2$. In particular, we must have
\[
2^s-p^{fk}=1
\]
for some integer $s\leq t$. By the Catalan--Mih\u{a}ilescu theorem (see the main theorem of \cite{26}) we must have $fk=1$. Hence $f=1$, $k=1$, and $n=2$. Hence
\[
S\cong \PSL_2(p).
\]
Similarly, writing $q'=r^{f'}$ we get:
\[
T\cong \PSL_2(r),
\]
where \(r+1\) is a power of \(2\). Now let \(K\leq S\) and \(N\trianglelefteq K\) be such that
\[
K/N\cong \operatorname{PSL}_{2}(7).
\]
Since $\operatorname{PSL}_{2}(7)$ is non-solvable, the group $K$ is also non-solvable. By the classification of the subgroups of $\PSL_2(q)$ (see \cite[Theorem~8.27]{18}), it follows that $K$ must itself be simple, and hence $N=1$. Therefore,
\[
\operatorname{PSL}_{2}(7)\leq \PSL_2(p).
\]
Applying \cite[Theorem~8.27]{18} once again, we conclude that $p=7$. Thus $S \simeq \operatorname{PSL}_{2}(7)$. We distinguish two cases separately:

\medskip

\noindent
\textbf{The case r=7.}
In this case 
\[
(B,\cdot)\cong \operatorname{PSL}_2(7)\times \operatorname{PSL}_2(7),
\]
Since $|\operatorname{PSL}_{2}(7)|=2^3\cdot 3\cdot 7$, we have $v_3(|B|)=2$. Assume that
\[
3\mid [(B,+):(B,+)'].
\]
Then by \cite[Lemma~2.1]{16} there exists a characteristic subgroup $K$ of $(B,+)$ such that
\[
(B,+)/K\cong C_3^m,
\qquad m\in \{1,2\}.
\]
Hence $(B,\cdot)$ admits a subgroup of index $3^m$, which is impossible by Lemma \ref{index}. Therefore
\[
3\nmid [(B,+):(B,+)'].
\]
Thus every Sylow $3$-subgroup of $(B,+)$ is contained in $(B,+)'$. By \cite[Theorem~9.1]{18}, the derived subgroup $(B,+)'$ is nilpotent. Hence $(B,+)$ has a normal Sylow $3$-subgroup. Let $P_3$ denote the Sylow $3$-subgroup of $(B,+)$. Since \(P_3\) is the unique Sylow \(3\)-subgroup of \((B,+)\), it is characteristic in \((B,+)\). Therefore its centralizer $C_{(B,+)}(P_3)$ is also characteristic in $(B,+)$. Thus, by \cite[X.19, Corollary, p.~339]{19},
\[
(B,+)/C_{(B,+)}(P_3)\leq \Aut(P_3).
\]
We now show that $\Aut((B,+)/C_{(B,+)}(P_3))$ is solvable. If $P_3$ is cyclic, then
\[
\Aut(P_3)\cong \Aut(C_9)\cong C_6,
\]
and therefore $(B,+)/C_{(B,+)}(P_3)$ is cyclic and $\Aut((B,+)/C_{(B,+)}(P_3))$ is solvable. Suppose instead that
\[
P_3\cong C_3\times C_3.
\]
Then
\[
\Aut(P_3)\cong \GL_2(3).
\]
Thus $(B,+)/C_{(B,+)}(P_3)$ is a subgroup of $\GL_2(3)$. Using GAP \cite{15}, one checks that $\Aut(H)$ is solvable for every subgroup $H\leq \GL_2(3)$. Thus $\Aut((B,+)/C_{(B,+)}(P_3))$ is solvable. In any case, by Lemma \ref{Key} there exists $k\in \mathbb N$ such that
\[
(B,\cdot)^{(k)}\leq C_{(B,+)}(P_3).
\]
Since $(B,\cdot)$ is perfect, we conclude that
\[
C_{(B,+)}(P_3)=B,
\]
which implies that $P_3\leq Z(B,+)$. Thus $P_3$ is abelian. By \cite[Theorem~14.2]{18} we have
\[
P_3\cap (B,+)'\cap Z(B,+)=1.
\]
Since $P_3\leq (B,+)'$, it follows that
\[
P_3\cap Z(B,+)=1,
\]
which contradicts the fact that $P_3$ is central.

\medskip

\noindent
\textbf{The case r > 7.}
Let \(R\) be a Sylow \(r\)-subgroup of \((B,+)\). Since \(r>7\) and \(r+1\)
is a power of \(2\), the prime \(r\) is the largest prime divisor of
$|B|$, and it divides \(|B|\) exactly once. Hence $|R|=r.$ Since \((B,+)\) is supersolvable, its Sylow subgroup corresponding to the largest prime divisor of \(|B|\) is normal. Thus $R$ is a characteristic subgroup of $(B,+)$. We claim that
\[
R\leq (B,+)'.
\]
Suppose not. Then \(r \mid |(B,+)/(B,+)'|\).  By \cite[Lemma~2.1]{16} there is a characteristic \(K\) of \((B,+)\) such that
\[
(B,+)/K\cong C_r.
\]
Since \(K\) is characteristic in \((B,+)\), it is invariant under all
\(\lambda_b\), \(b\in B\). Hence \(K\) is a left ideal of \(B\). Therefore $(K,\cdot)$
is a subgroup of \((B,\cdot)\) of index \(r\).
But $G:=\operatorname{PSL}_2(7)\times \operatorname{PSL}_2(r)$
has no subgroup of index \(r\). Suppose otherwise, and let \(M\leq G\)
with \([G:M]=r\). Suppose, by contradiction, that \(M\leq G\) and \([G:M]=r\). Write
\[
S=\operatorname{PSL}_2(7),\qquad T=\operatorname{PSL}_2(r).
\]
Identifying \(S\) and \(T\) with \(S\times 1\) and \(1\times T\), respectively, we have
\[
[S:S\cap M]=[SM:M]
\]
and
\[
[T:T\cap M]=[TM:M].
\]
Both indices divide \([G:M]=r\). Hence each of them is either \(1\) or \(r\).
They cannot both be \(1\), because otherwise \(S\leq M\) and \(T\leq M\), and
therefore \(G=ST\leq M\), contradicting \(M<G\).
Thus either \(S\) has a subgroup of index \(r\), or \(T\) has a subgroup of
index \(r\). This is impossible. Indeed, \(S=\operatorname{PSL}_2(7)\) has order \(168\), so it
has no subgroup of index \(r>7\). On the other hand, by the classification of
the subgroups of \(\operatorname{PSL}_2(r)\) (see \cite[Theorem~8.27]{18}), the smallest index of a proper subgroup of
\(\operatorname{PSL}_2(r)\) is \(r+1\). Hence \(T=\operatorname{PSL}_2(r)\) has no subgroup of index \(r\).
This contradiction proves that \(G\) has no subgroup of index \(r\).  Therefore no such \(K\) exists, and we have proved that $R\leq (B,+)'.$ Now \(R\) is characteristic in \((B,+)\). Hence the lambda action induces a homomorphism
\[
\lambda_R: (B,\cdot) \longrightarrow \operatorname{Aut}(R), \qquad b \mapsto \lambda_{b}|_{R}.
\]
Since \(R\cong C_r\), we have $\operatorname{Aut}(R)\cong C_{r-1},$ which is abelian. On the other hand, \((B,\cdot)\) is perfect.
Hence every homomorphism from \((B,\cdot)\) to an abelian group is trivial. Therefore $\lambda_b(a)=a$
for every \(b\in B\) and every \(a\in R\). Thus
\[
b\cdot a=b+a
\]
for every \(b\in B\) and every \(a\in R\). Applying the same argument to the opposite skew brace, we also obtain
\[
b\cdot a=a+b
\]
for every \(b\in B\) and every \(a\in R\). Hence
\[
b+a=a+b
\]
for all \(b\in B\) and all \(a\in R\). Therefore
$R\leq Z(B,+).$ By \cite[Theorem~14.2]{18} we have
\[
R \cap (B,+)'\cap Z(B,+)=1.
\]
Since $R\leq (B,+)'$, it follows that
\[
R\cap Z(B,+)=1,
\]
which contradicts the fact that $R$ is central.
Therefore \((B,+)\) cannot be supersolvable.

\end{proof}

\section*{Acknowledgements}
The author is grateful to Andrea Loi and Matteo Bechere for several useful
discussions.

\end{document}